\documentclass[12pt]{amsart}

\usepackage[all]{xy}
\newtheorem{thm}[subsection]{Theorem}

\newtheorem{lem}[subsection]{Lemma}

\newcommand{\CC}{{\mathbb C}}
\newcommand\si{\sigma}

\begin{document}

\baselineskip=16pt
 
\title[Equivariant vector bundles on complete symmetric
varieties]{Equivariant vector bundles on complete symmetric varieties of minimal rank}

\author[I. Biswas]{Indranil Biswas}

\address{Tata Institute of Fundamental Research,
Homi Bhabha Road, Mumbai 500004, India}

\email{indranil@math.tifr.res.in}

\author[S. S. Kannan]{S. Senthamarai Kannan}

\address{Chennai Mathematical Institute, Plot H1, Sipcot IT Park,
Siruseri, Kelambakam, Chennai 603103, India}

\email{kannan@cmi.ac.in}

\author[D. S. Nagaraj]{D. S. Nagaraj}

\address{Institute of Mathematical Sciences, C.I.T. Campus, Taramani, 
Chennai 600113, India}

\email{dsn@imsc.res.in}



\date{}

\begin{abstract}
Let $X$ be the wonderful compactification of a complex symmetric space 
$G/H$ of minimal rank. For a point $x\,\in\, G$, denote by $Z$ be the closure of $BxH/H$ 
in $X$, where $B$ is a Borel subgroup of $G$. The universal cover of $G$
is denoted by $\widetilde{G}$. Given a $\widetilde{G}$ equivariant vector bundle $E$
on $X,$ we prove that $E$ is nef (respectively, ample) if and only if its 
restriction to $Z$ is nef (respectively, ample). Similarly, $E$ is trivial
if and only if its restriction to $Z$ is so.
\end{abstract}

\maketitle

\section{Introduction} \setcounter{page}{1}

Let $\sigma$ be an involution of a semisimple adjoint type algebraic group $G$ over
$\mathbb{C}$, and let $H\,=\, G^\sigma$ be the corresponding fixed point locus.
De Concini and Procesi constructed a smooth projective variety
$$
X\,=\, \overline{G/H}
$$
equipped with an action of $G$, that contains an open dense $G$--orbit 
$G/H$ \cite{DP}. This $X$ is known 
as the wonderful compactification of the symmetric space $G/H$.

Richardson and Springer described the $B$--orbits in $G/H$ 
in terms of the combinatorics of the Weyl group $W,$ where $B$ is a Borel subgroup
of $G$ (see \cite{RS}). The rank of $G/H$ is defined by 
Panyushev \cite{Pa} and Knop \cite{Kn}. 
The minimal rank symmetric spaces were introduced by Brion \cite{Br}. 
Brion and Joshua have studied the geometry of the closures in $X$ of the $B$--orbits
in $G/H,$ whenever $G/H$ is of minimal rank \cite{BJ}. Tchoudjem has 
also studied the closures in $X$ of the $B$--orbits in $G/H$, 
whenever $G/H$ is of minimal rank \cite{Tc}.

This paper deals with the restriction of equivariant vector bundles 
on $X$ to some natural class of subvarieties of $X$, like $B$--orbit closures. 

Let $\widetilde{G}$ be the simply connected covering of $G$. The action of $G$ on $X$
produces an action of $\widetilde{G}$ on $X$ using the natural projection
$\widetilde{G}\, \longrightarrow\, G$. Given an algebraic vector
bundle $E$ on $X$, we can get a class of vector bundles on $X$ by
pulling back $E$ using the automorphisms of $X$ given by the action of
$G$. It can be shown that the isomorphism classes of these pullbacks remain constant
if and only if $E$ admits a $\widetilde{G}$--equivariant structure (meaning
the action of $\widetilde{G}$ on $X$ admits a lift to an action on $E$).

We prove the following (see Theorem \ref{thm1}):

\begin{thm}\label{th1}
Assume that $G/H$ is of minimal rank.
Fix a point $x\,\in\, G.$ Let $Z$ be the closure of $BxH/H$ 
in $X$. Let $E$ be a $\widetilde{G}$ equivariant vector bundle on $X$. Then, 
$E$ is nef (respectively, ample) if and only if the 
restriction of $E$ to $Z$ is nef (respectively, ample). Similarly,
$E$ is trivial if and only if its restriction to $Z$ is trivial.
\end{thm}

In \cite{HMP}, a similar result is proved for vector bundles on toric varieties.

Before stating the next result, we recall that for the conjugation action 
of $\widetilde{G}$ on itself, Steinberg proved that for a
maximal torus $T$ of $G$, the restriction homomorphism
$$
\mathbb{C}[\widetilde{G}]^{\widetilde{G}}\,\longrightarrow\, \mathbb{C}
[\widetilde{T}]^{W(G,T)}
$$
is an isomorphism, where
$\widetilde{T}$ is the inverse image of $T$ in 
$\widetilde{G}$ and $W(G,T)$ is the Weyl group of $G$ with respect to $T$
\cite{St1}. Hence, we have the Steinberg map
$$
\tau\,:\, \widetilde{G}\,\longrightarrow \,\widetilde{T}/{W(G,T)}\,=\,{\mathbb A}^{n}\, .
$$

Let $c$ be a Coxeter element in the Weyl group $W(G,T)$, and let $F$ 
be the  fiber of the Steinberg map $\tau$ containing a 
representative $n_{c}$ of $c$ in $N_{\widetilde{G}}(\widetilde{T})$. Let 
$F^{\prime}$ be the image of $F$ in $G$. Set ${\mathcal Z} \,= \, Z_1\bigcup Z_2$,
where $Z_1$ is the closure of $F^{\prime}$ in 
the wonderful compactification $\overline{G}$ of $G$, and
$Z_2$ is the unique closed $G\times G$ orbit in $\overline{G}$ \cite{DP}.

The group $\widetilde{G}\times \widetilde{G}$ acts on $\overline{G}$ which factors
through the action of $G\times G$ on $\overline{G}$. Given an algebraic vector bundle
$E$ on $\overline{G}$, the isomorphism classes of its translates by the elements of
$G\times G$ remain constant if and only if $E$ admits a $\widetilde{G}\times
\widetilde{G}$--equivariant structure.

We also prove the following (see Theorem \ref{thm2}):

\begin{thm}\label{th2}
Let $E$ be a $\widetilde{G}\times \widetilde{G}$ equivariant vector bundle on
$\overline{G}$. Then, $E$ is nef (respectively, ample) if and only 
if the restriction of $E$ to ${\mathcal Z}$ is nef (respectively, ample). Similarly,
$E$ is trivial if and only if $E\vert_{\mathcal Z}$ is trivial.
\end{thm}

\section{Preliminaries}

\subsection{Lie algebras and Algebraic groups}

In this subsection we recall some basic facts and notation on Lie algebras
and algebraic groups (see \cite{Hu},\cite{Hu1} for details). Throughout $G$ denotes a
semisimple adjoint-type algebraic group over the field $\mathbb{C}$ of complex numbers. In
particular, the center of $G$ is trivial. For a maximal torus $T$ of $G$, the group of all
characters of $T$ will be denoted by $X(T)$. The normalizer of $T$ in $G$ will be denoted
by $N_{G}(T)$, while $$W(G,T)\ :=\, N_{G}(T)/T$$ is the Weyl group of $G$ with respect to $T$.
Let $R \,\subset \,X(T)$ be the root system of $G$ with respect to $T$.
For a Borel subgroup $B$ of $G$ containing $T$, let 
$R^{+}(B)$ denote the set of positive roots determined by $T$ and $B$. Further,
$$
S \,=\, \{\alpha_1,\cdots,\alpha_n\}
$$ 
denotes the set of simple roots in $R^{+}(B)$. For $\alpha \,\in\, R^{+}(B)$, let 
$s_{\alpha} \,\in \, W(G,T)$ be the 
reflection corresponding to $\alpha$. The Lie algebras of
$G$, $T$ and $B$ will be denoted by $\mathfrak{g}$, $\mathfrak{t}$ and $\mathfrak{b}$
respectively. The dual of the real form $\mathfrak{t}_{\mathbb R}$ of $\mathfrak{t}$ is
$X(T)\otimes \mathbb{R}\,=\, Hom_{\mathbb{R}}(\mathfrak{t}_{\mathbb{R}},\, \mathbb{R})$.
 
The positive definite $W(G,T)$--invariant form on $Hom_{\mathbb{R}}(\mathfrak{t}_{\mathbb{R}},
\, \mathbb{R})$ induced by the Killing form on $\mathfrak{g}$ is denoted by $(-\, ,-)$. 
We use the notation $$\langle \nu\, , \alpha \rangle \,:= \,\frac{2(\nu,
\alpha)}{(\alpha,\alpha)}\, .$$ In this setting one has the Chevalley basis 
$$\{x_{\alpha}, h_{\beta} \,\mid\, \alpha \,\in\,
R,\, ~ \beta \,\in\, S\}$$ of $\mathfrak{g}$ determined by $T$.
For a root $\alpha$, we denote by 
$U_{\alpha}$ (respectively, $\mathfrak{g}_{\alpha}$) the one--dimensional
$T$ stable root subgroup of $G$ (respectively, the 
subspace of $\mathfrak{g}$) on which $T$ acts through the 
character $\alpha$. 

Let $\sigma$ be an algebraic automorphism of $G$ of order two. Let $H\, =\, G^\sigma$
be the subgroup consisting of all fixed points of $\sigma$ in $G$. The connected component
of $H$ containing the identity element will be denoted by $H^0$.
We refer to \cite{Ri} and \cite{RS} for following facts. 

A torus $T^{\prime}$ of $G$ is said to be $\sigma$-anisotropic if 
$\si(t) \,= \,t^{-1}$ for every $t\,\in\, T^{\prime}$.
Recall that the rank of $G/H$ is defined to be
the dimension of a maximal dimensional anisotropic torus.

If $T$ is a $\sigma$-stable maximal torus of $G,$ then $\sigma$ induces an automorphism of 
$X(T)$ of order two. Note that we have $\sigma(R)\,=\,R$. Further, one has
$T\,=\,T_{1}T_{2},$ where $T_{1}$ is a torus such that $\sigma(t)\,=\,t$ 
for every 
$t\,\in\, T_{1},$ and $T_{2}$ is a $\sigma$-anisotropic torus. Clearly $T_{1}\bigcap T_{2}$ is finite. 
Hence, we have ${\rm rank}(G/H)\,\geq\, {\rm rank}(G)-{\rm rank}(H)$.

Throughout, we assume that $G/H$ is of minimal rank, or in other words
$$
{\rm rank}(G/H)\,=\,{\rm rank}(G)-{\rm rank}(H)\, .
$$
We refer to \cite{Br} and \cite{Kn1} for facts about minimal rank.

The following lemma may be known, but for the sake of completeness 
we provide a proof here. 
 
\begin{lem}\label{lem1}
\mbox{}
\begin{enumerate}
\item
Any two $\sigma$ stable maximal tori of $G$ are conjugate by an element of the
connected component $H^{0}$ of $H$ containing the identity element.

\item
Any maximal torus $S$ of $H^{0}$ is contained in a unique maximal torus $T$ 
of $G$. Further, this $T$ is $\sigma$ stable.
\item
Any Borel subgroup $Q$ of $H^{0}$ is contained in a $\sigma$ stable Borel
subgroup $B$ of $G$. Further, the Borel subgroup of $G$
containing $Q$ is unique.
\end{enumerate}
\end{lem}

\begin{proof}
Proof of (1). Let $T_{1}$ and $T_{2}$ be two $\sigma$ stable maximal tori in $G$.
Define $$S_{i}\, :=\,(T_{i}\cap H)^{0}\, ,$$ $i\,=\, 1\, ,2$. Since $G/H$ is of minimal rank, 
$S_{1}$ and $S_{2}$ are maximal tori in $H^{0}$. Hence, there is an element $h\,\in\, H^{0}$
such that $hS_{1}h^{-1}\,=\,S_{2}$. Also, $T_{i}\,=\,C_{G}(S_{i})$
(see \cite[p. 295, Lemma 5.3 and Lemma 5.4]{Ri}).

Proof of (2). Take $T\,=\,C_{G}(S)$.

Proof of (3). We will first prove the existence of a stable Borel subgroup containing $Q$.

By \cite[p. 51, Lemma 7.5]{St2}, there is a 
$\sigma$ stable Borel subgroup $B^{\prime}$ of $G$. 
By \cite[p. 295, Lemma 5.1]{Ri}, the intersection
$(B^{\prime}\bigcap H)^{0}$ is a Borel subgroup of $H^{0}$. Hence, there is a 
$h\,\in\, H^{0}$ such that $Q\,=\,h(B^{\prime}\bigcap H)^{0}h^{-1}$. Now take 
$B\,=\,hB^{\prime}h^{-1}$.

To prove the uniqueness of $B,$ let $B_{1}$ be a Borel subgroup of $G$ containing
$Q$. As shown above, there is a $\sigma$ stable Borel subgroup $B$ of $G$ containing $Q$.
Choose a maximal torus $S$ of $H^{0}$ lying in $Q$. From part (2) of the lemma we know
that $T\,=\,C_{G}(S)$ is the unique maximal torus of $G$ containing $S$. Hence, $T$ is 
contained in both $B_{1}$ and $B$. Thus, there is a $w\,\in\, W(G,T)$
such that $wBw^{-1}\,=\,B_{1}$. 

We now prove that $R^{+}(B_{1})\,=\,R^{+}(B)$.
Let $\alpha \,\in\, R^{+}(B)\setminus R^{+}(B)^{\sigma}$. Then the $\sigma$ invariant 
vector $x_{\alpha}+\sigma(x_{\alpha})$ is in the Lie algebra of $Q$. 
Hence, $x_{\alpha}+\sigma(x_{\alpha})$ is in the Lie algebra of $B_{1}$.
Thus, both $\alpha$ and $\sigma(\alpha)$ are in 
$R^{+}(B_{1})\setminus R^{+}(B_{1})^{\sigma}$. Hence, we have 
$$
R^{+}(B)\setminus R^{+}(B)^{\sigma}\,=\,R^{+}(B_{1})\setminus R^{+}(B_{1})^{\sigma}\, .
$$ 

Now, let $\alpha\,\in\, R^{+}(B)^{\sigma}$. 
We will show that $\sigma$ acts trivially on $U_{\alpha}$. Let $T_{\alpha} \,\subset\, T$
be the connected component, containing the identity element, of the kernel of $\alpha$.
Consider the restriction of $\sigma$ to $C_{G}(T_{\alpha})$.
Let $C^{\prime}$ be the commutator subgroup of $C_{G}(T_{\alpha})$.
If the action of $\sigma$ on $U_{\alpha}$ is not trivial, it follows that 
there is a one-dimensional $\sigma$ stable anisotropic torus $S^{\prime}$ in 
$C^{\prime}$. Let $T_{1}\,=\,S^{\prime}T_{\alpha}$. Then we have 
$T_{1}^{\sigma}\,=\,T_{\alpha}^{\sigma}$. 
Hence by \cite[Lemma 5.4]{Ri} we have $T_1\,= \,C_G(T_{\alpha}^{\sigma})$.
But this contradicts the fact that $T_{\alpha}^{\sigma}$ is a singular torus.
Hence we have $U_{\alpha}\subset (B)^{\sigma} = Q \subset B_1$.

Thus, we have shown that $R^{+}(B)\,=\,R^{+}(B_{1})$. Hence, we have $B\,=\,B_{1}$.
This completes proof.
\end{proof}

\subsection{Nef vector bundle}\label{sec2.1}

Let $E$ be an algebraic vector bundle over a complex projective variety $Y$. Let
${\mathbb P}(E)$ denote the associated projective bundle over $Y$ whose fiber
over any point $y\, \in\, Y$ is the space of all one-dimensional quotients of
the fiber $E_y$ of $E$ over $y$. The line bundle over ${\mathbb P}(E)$ whose fiber over any
one-dimensional quotient is the one-dimensional quotient itself, will be denoted
by ${\mathcal O}_{{\mathbb P}(E)}(1)$.

A line bundle $L$ over $Y$ is called \textit{nef} if for every pair $(C\, ,\varphi)$, where
$C$ is an irreducible smooth complex projective curve and $\varphi\, :\, C\,
\longrightarrow\, Y$ is a morphism, the degree of the pullback $\varphi^*L$ is nonnegative.
A vector bundle $E\, \longrightarrow\, Y$ is called \textit{nef} if the above line bundle
${\mathcal O}_{{\mathbb P}(E)}(1)$ over ${\mathbb P}(E)$ is nef.

\section{Restriction of equivariant vector bundles to $B$-orbit closure}

Let $T$ be a $\sigma$ stable maximal torus of $G.$  Let $B$ be a Borel subgroup of $G$ containing
$T$ such that for any root $\alpha\,\in\, R^{+}(B)$,
either $\sigma(\alpha)\,=\,\alpha$ or $\sigma(\alpha)\,\in\, -R^{+}(B).$ 

Let
$$
X \,:= \,\overline{G/H}
$$
be the wonderful compactification of the symmetric space
$G/H$ constructed in \cite{DP}. Let $Z$ be the 
closure in $X$ of the $B$--orbit of a point in $G/H$. 

Let $P$ be the parabolic subgroup of $G$ containing $B$ such that the 
$G/P$ is the unique closed $G$ orbit in $X$ (see \cite{DP}). In this 
case, $\sigma(P)$ is opposite to $P$ and $P\bigcap \sigma(P)$ is the Levi
subgroup $L$ of $P$. Let $R(L)$ denote the roots of $L$ with respect to $T$. 

The following lemma is about a $B$--orbit in $G/H$. We refer to \cite{RS}
for information on $B$--orbit closures in $G/H$. For any algebraic group
acting on variety, it is well known that there is always a closed orbit.
For instance, any orbit of minimal dimension is 
closed (see, \cite[ p. 60. Proposition]{Hu1}).

\begin{lem}\label{lem3}
Let $x\,\in\, G$ be such that $B\cdot xH/H$ is closed in $G/H.$ Then, 
$x^{-1}Bx$ is $\sigma$ stable and there is a $w\,\in\, W(G,T)$
such that $B\cdot xH/H\,=\, B\cdot wH/H$.
\end{lem}

\begin{proof}
Let $Q\,:=\,(x^{-1}Bx\bigcap H)^{0}$.
Since $B\cdot xH/H$ is closed in $G/H,$ this $Q$ is a Borel subgroup of $H^{0}.$
Further, we have $Q\,\subset\, x^{-1}Bx.$ Hence, by Lemma \ref{lem1}, 
$x^{-1}Bx$ is $\sigma$ stable. 

Now, let $S\,=\,(T\bigcap H)^{0}$. 
Since $G/H$ is of minimal rank, this $S$ is a maximal torus in $H^{0}$, and hence 
we choose a Borel subgroup $Q^{\prime}$ of $H^{0}$ containing $S.$ Thus, 
there is a $h\,\in\, H^{0}$ such that $hQh^{-1}\,=\,Q^{\prime}$. Consequently, 
$hx^{-1}Bxh^{-1}$ is a Borel subgroup of $G$ containing $T$. Thus, there is a
$w\in W(G,T)$ and a $b\in B$ such that 
$xh^{-1}\,=\, bw$, and we have $B\cdot xH/H\,=\,B\cdot wH/H$.
\end{proof}

An interesting fact in case of minimal rank is the following uniqueness of the
closed $B$--orbit (see, \cite[p. 1788, Proposition 2.2]{Re}).

\begin{lem}\label{lem4}
There is a unique closed $B$--orbit in $G/H$.
\end{lem}

\begin{proof}
Clearly, there is a minimal dimensional $B$--orbit in $G/H$ 
and it is closed. For its uniqueness, let $Bx_{1}H/H$ and let $Bx_{2}H/H$ be two 
closed $B$--orbits in $G/H$. Then, by Lemma \ref{lem3}, there are $w_{1}$ and $w_{2}$ 
in $W$ such that $B\cdot x_{i}H/H\,=\, B\cdot w_{i}H/H$ for $i\,=\,1\, , 2.$

Let $S\,=\,(T\bigcap H)^{0}.$ Set  $B_{i}\,:=\,w_{i}^{-1}Bw_{i},$
and $Q_{i}\,=\,(B_{i}\bigcap H)^{0}$ for $i\,=\,1\, ,2.$ Both $Q_{1}$ and $Q_{2}$ are Borel subgroups of $H^{0}$
containing $S.$ Therefore, there is a $\phi \,\in\, W(H^{0}, S)$ such that 
$\phi Q_{1} \phi^{-1}\,=\,Q_{2}.$ Hence both $\phi B_{1} \phi^{-1}$ and $B_{2}$
are Borel subgroups of $G$ containing $Q_{2}.$ By Lemma \ref{lem1}, we
have $\phi B_{1} \phi^{-1}\,=\,B_{2},$ and hence $w_{1}\,=\,w_{2}\phi.$
Thus $Bx_{1}H/H\,=\,Bx_{2}H/H$.
\end{proof}

We now recall from \cite{BJ} a result of Brion and Joshua.

\begin{lem}[{\cite[p. 482, Lemma 2.1.1]{BJ}}]\label{lem5}
Let $Y$ be the closure of $TH/H$ in $X,$
and let $z$ denote the unique $B$--fixed point in $X.$
Then, every $T$ stable curve in $X$ is one of the following:
\begin{enumerate}
\item
There is a positive root $\alpha \,\in\, R^{+}(B)\setminus R^{+}(L)$ and an element 
$\phi\,\in\, W(G,T)$ such that $\phi(C)\,=\, C_{\alpha}\,=
\,\overline{U_{\alpha}s_{\alpha}z}.$ 
In this case $\alpha$ and $\sigma(\alpha)$ are orthogonal, and 
$s_{\alpha}s_{\sigma(\alpha)}$ is in $W(H^{0}, (T\bigcap H)^{0}).$ 

\item There is a restricted root $\gamma\,=\,\alpha-\sigma(\alpha)$ and an element
$\phi \,\in \,W(G,T)$ such that $\phi(C)\,=\,C_{z, \gamma}$, where $C_{z, \gamma}$ is
the unique $T$--stable curve containing $z$ and on which $T$ acts through the character
$\gamma$. Moreover, the curve $C_{z, \gamma}$ lies in $Y.$
\end{enumerate}
\end{lem}

\begin{lem}\label{lem6}
Take $x\,\in \,G,$ and let $Z$ be the closure of $BxH/H$ in $X.$
Then every irreducible $T$ stable curve in $X$ lies in $W(G,T)\cdot Z.$
\end{lem}

\begin{proof}
Note that the closure of $B\cdot xH/H$ in $G/H$ contains a closed $B$
orbit. Therefore we assume that $B\cdot xH/H$ is the unique closed $B$ orbit in $G/H.$

By Lemma \ref{lem3}, there is an element $w\,\in\, W(G,T)$ such that $B\cdot xH/H
\,=\,B\cdot wH/H.$ Let $C$ be an irreducible $T$ stable curve in $X.$ By Lemma
\ref{lem5},
\begin{itemize}
\item either there is a positive root $\alpha \,\in\, R^{+}(B)\setminus R(L)$ and a 
$\phi \,\in\, W(G,T)$ such that $\phi (C)\,=\, C_{\alpha}\,=\,
\overline{U_{\alpha}s_{\alpha}z},$

\item or there is a restricted root $\gamma$ and a $\phi \,\in\, W(G,T)$ 
such that $\phi(C)\,=\,C_{z, \gamma}.$ 
\end{itemize}

Recall that $Y\,=\,\overline{TH/H}$ and $S =\,(T\cap H)^0.$ Now, since $s_{\alpha}s_{\sigma(\alpha)}\,\in\, W(H^{0}, S),$  
and $z\,\in\, Y$ (see, Lemma \ref{lem5} (2)), we have 
$$s_{\alpha}s_{\sigma(\alpha)}\cdot z\,\in\, Y\, .$$
Hence, 
$ws_{\alpha}s_{\sigma(\alpha)}\cdot z\in w\cdot Y\,=\,\overline{TwH/H}.$
Since $\alpha$ and $\sigma(\alpha)$ are orthogonal, $s_{\alpha}s_{\sigma(\alpha)}(\alpha)
\,=\,-\alpha.$ Hence, either $w(\alpha)$ is positive or 
$ws_{\alpha}s_{\sigma(\alpha)}(\alpha)\,=\,w(-\alpha)$ is positive. 
Further, $s_{\alpha}s_{\sigma(\alpha)}\,\in\, W(H^{0}, S).$
Hence $BwH/H\,=\,Bws_{\alpha}s_{\sigma(\alpha)}H/H.$

Now, if $w(\alpha)$ is positive, then $U_{w(\alpha)}ws_{\alpha}s_{\sigma(\alpha)}\cdot z$
is contained in $\overline{BwH/H}$. Hence, $$ws_{\sigma(\alpha)}(C_{\alpha})\,=\,
\overline{ws_{\sigma(\alpha)}U_{\alpha}s_{\alpha}\cdot z}\,=\,
\overline{U_{w(\alpha)}ws_{\alpha}s_{\sigma(\alpha)}\cdot z}$$ is contained in $\overline{BwH/H}.$

If $ws_{\alpha}s_{\sigma(\alpha)}(\alpha)\,=\,w(-\alpha)$ is positive, then
$ws_{\alpha}(C_{\alpha})\,=\, \overline{U_{w(-\alpha)}w\cdot z}$ is contained in 
$\overline{BwH/H}.$

Thus, in either case, the curve $C_{\alpha}$ lies in $W(G,T)\cdot Z.$

Since $C_{z, \gamma}\,\subset\, Y,$ we have $w(C_{z, \gamma})\,\subset \,\overline{TwH/H}.$
Hence, both type of curves in Lemma \ref{lem5} lie in the union of the $W(G,T)$ 
translates of $\overline{BwH/H}\,=\,\overline{BxH/H}.$
This completes the proof. 
\end{proof}

{\bf Notation:}\, Let $G$ be a semi-simple adjoint group over the field $\CC$ of complex
numbers as above, and let $\widetilde{G}$ be its universal cover. 
For a maximal torus $T$ in $G$, we denote its inverse
image in $\widetilde{G}$ by $\widetilde{T}.$

Note that $\widetilde{G}$ acts on $X$ and hence we can consider $\widetilde{G}$
equivariant vector bundles on $X.$

\begin{thm}\label{thm1}
Fix a point $x\,\in\, G.$ Let $Z$ be the closure of $BxH/H$ in $X,$ where
$B$ is a $\sigma$ stable Borel subgroup of $G.$ Let $E$ be a 
$\widetilde{G}$ equivariant vector bundle on $X.$
Then, $E$ is nef (respectively, ample) if and only if the 
restriction of $E$ to $Z$ is nef (respectively, ample). Similarly,
$E$ is trivial if and only if its restriction to $Z$ is trivial.
\end{thm}

\begin{proof}
Since the restriction of a nef or ample or trivial
vector bundle to a subvariety is nef or ample or trivial
respectively, we have only to prove the ``if'' part of the theorem. 

First assume that the restriction $E\vert_Z$ is nef. We need to show that
for any irreducible closed curve $C$ in $\mathbb{P}(E),$
the degree of the line bundle $\mathcal{O}_{\mathbb{P}(E)}(1)\vert_C$
is nonnegative, where $\mathcal{O}_{\mathbb{P}(E)}(1)\,\longrightarrow\,
\mathbb{P}(E)$ is the line bundle defined in Section \ref{sec2.1}.

Let $Y(\widetilde{T})$ denote the group of all one-parameter subgroups of
$\widetilde{T},$ where $\widetilde{T},$ as before, is the inverse image in
$\widetilde{G}$ of a $\sigma$ stable maximal torus $T$ of $G$ lying in $B.$
Choose a $\mathbb Z$--basis $\{\lambda_{1}\, , \lambda_{2}\, , \cdots\, , \lambda_{n}\}$ of 
$Y(\widetilde{T}).$

Let $\widetilde{C}$ be an irreducible closed
curve in the projective bundle $\mathbb{P}(E)$ over $X.$ If the image of $C$ in $X$ is a
point, then the degree of $\mathcal{O}_{\mathbb{P}(E)}(1)$ 
restricted to $\widetilde{C}$ is positive, because $\mathcal{O}_{\mathbb{P}(E)}(1)$
is relatively ample. Hence we can assume that image of 
$\widetilde{C}$ in $X$ is a curve $C$. Let $\widetilde{C}_1$ be the flat limit of
$\lambda_{1}(t)\widetilde{C}$ as $t$ goes to zero (i.e., the one
dimensional cycle corresponding to the limit point in
the Hilbert Scheme of $\mathbb{P}(E)$). Then $\widetilde{C}_1$ is a $1$-dimensional
cycle in $\mathbb{P}(E)$ linearly equivalent to $\widetilde{C},$ and the image 
$C_1$ of $\widetilde{C}_1$ in $X$ is invariant under $\lambda_{1}.$ Inductively,
define $\widetilde{C}_i$ to
be the flat limit of $\lambda_{i}(t)\widetilde{C}_{i-1}$ as $t$ tends to zero, where 
$2\,\leq\, i \,\leq\, n.$ Then $\widetilde{C}_i$ is linearly equivalent to
$\widetilde{C},$ and the image
$C_i$ of $\widetilde{C}_i$ in $X$ is invariant under the action on $X$
of the sub-torus of $T$
generated by the images of $\{\lambda_{1}\, , \lambda_{2}\, , \cdots\, , \lambda_{i}\}.$ 

In particular, $\widetilde{C}_n$ is linearly equivalent to $\widetilde C$, and every
irreducible component of $\widetilde{C}_n$ lies in the preimage of the $T$ invariant curve
$C_n\,\subset\, X$. But $C_n$ can be conjugated to a curve in $Z$ (see Lemma \ref{lem6}),
hence, by our assumption, $E|_{C_n}$ is nef. Therefore, the degree of the line
bundle $\mathcal{O}_{\mathbb{P}(E)}(1)|_{\widetilde C}$
is nonnegative (recall that $\text{degree}(\mathcal{O}_{\mathbb{P}(E)}(1)|_{\widetilde C})
\,=\, \text{degree}(\mathcal{O}_{\mathbb{P}(E)}(1)|_{{\widetilde C}_n})$). 
This proves that $E$ is nef.

Next assume that $E\vert_Z$ is ample.

For any positive integer $n,$ let ${\rm Sym}^n(E)$
denote the $n$-th symmetric power of the equivariant vector bundle $E$.
To prove that $E$ is ample, we first note that there are only finitely
many $T$ stable curves in $X$, and all of them lie in 
$W(G,T)\cdot Z$ (see Lemma \ref{lem6}). Thus the assumption implies that
${\rm Sym}^n(E)|_C$ is ample for any $T$ stable curve $C$ in $X$ and for any
$n \,\geq\, 1$.

Since line bundles on $X$ are equivariant for the $\widetilde{G}$
action on $X$, the vector bundles ${\rm Sym}^n(E)\otimes L$ are all
$\widetilde{G}$ equivariant vector bundles on $X$, where $L$ is any line
bundle on $X$. Fix an ample line bundle $L$ on $X$, and let $n$ be an
integer such that $n \,> \, {\rm degree}(L|_C)$ for every $T$ invariant curve
$C$ in $X$. Then it follows from the argument in the first part of the
proof of the theorem that ${\rm Sym}^n(E)\otimes L^{-1}|_Z$  
is nef, and hence ${\rm Sym}^n(E)\otimes L^{-1}$ is nef. This
implies ${\rm Sym}^n(E)$ is ample and hence $E$ is ample
(see, \cite[p. 67, Proposition 2.4]{Ha}).

Finally assume that $E\vert_Z$ is trivial.

Since $E\vert_Z$ is trivial, the dual $(E\vert_Z)^*\,=\, E^*\vert_Z$ is also
trivial. Note that a trivial vector bundle is nef. Therefore, from the first part
of the theorem we conclude that both $E$ and its dual $E^*$ are
nef. Therefore, by \cite[p. 311, Theorem 1.18]{DPS} the vector bundle $E$ admits a
filtration of holomorphic subbundles
$$
0\, =\, E_0\, \subset\, E_1 \, \subset\, \cdots \, \subset\,
E_{\ell-1} \, \subset\, E_\ell\,=\, E
$$
such that each successive quotient $E_i/E_{i-1}$, $1\, \leq\, i \, \leq\,\ell$, admits a
unitary flat connection. This implies that
$E$ is semistable and $c_j(E)\,=\, 0$ for all $j\, \geq\, 1$, where $c_j$ is
the rational Chern class. Now, by \cite[p. 40, Corollary 3.10]{Si} the vector
bundle $E$ admits a flat holomorphic connection. 

The variety $X$ is simply connected, because it is
unirational (see, \cite[p. 483, Proposition 1]{Se}). Therefore, any holomorphic
vector bundle on $X$ admitting a flat holomorphic connection is a
holomorphically trivial vector bundle. In particular, the vector bundle
$E$ is trivial.
\end{proof}

The proof of first two parts of the above theorem closely follows that
of \cite[p.610, Theorem 2.1]{HMP}.

\section{A special Steinberg fiber}

As before, $G$ be a semisimple adjoint group. 
Let $T$ be a maximal torus of $G$, $W(G,T)$ the Weyl group of $G$ with 
respect to $T$ and $B$ a Borel subgroup of $G$ containing $T.$
Let $\widetilde{G}$ be the simply connected covering of $G$, and let 
$\widetilde{T}$ (respectively, $\widetilde{B}$) be the 
inverse image of $T$ (respectively, $B$) in $\widetilde{G}$. 
Let $c$ be a Coxeter element in $W$. We fix a 
representative $n_{c}$ of $c$ in $N_{\widetilde{G}}(\widetilde{T})$.

\begin{lem}\label{lem7}
The homomorphism $\phi_{c}\,:\,\widetilde{T}\,\longrightarrow \,\widetilde{T}$ given by 
$\phi_{c}(t)\,=\,tn_{c}t^{-1}n_{c}^{-1}$ is surjective.
\end{lem}

\begin{proof}
It is enough to prove that the kernel of $\phi_{c}$ is finite.
We can choose a reduced expression $c\,=\,s_{\alpha_{1}}s_{\alpha_{2}}\cdots
s_{\alpha_{n}}$ for $c$ such that $\{\alpha_{1}\, , \alpha_{2}\, , \cdots\, ,
\alpha_{n} \}$ is the set of simple roots labeled in some ordering. 
Let $\beta_{i}\,=\,s_{\alpha_{1}}\cdots s_{\alpha_{i-1}}(\alpha_{i})$. 
Then, the set $\{\beta_{1}\, , \beta_{2}\, , \cdots \, ,\beta_{n} \}$
is the set of positive roots which are made negative by $c^{-1}$.

By \cite[p. 862, Lemma 2.1]{YZ}, we have $\omega_{i}-c(\omega_{i})\,=\,\beta_{i}$.
Now, let $t$ be an element of the kernel of $\phi_{c}$. Then, 
$\beta_{i}(t)\,=\,1$ for every $i\,=\,1\, , 2\, , \cdots\, , n$. 
Hence,
$$
{\rm kernel}(\phi_{c})\, \subset\, \bigcap_{i=1}^n {\rm kernel}(\beta_{i})\, .
$$
Since $\{\beta_{1}\, , \beta_{2}\, , \cdots\, , \beta_{n}\}$ is a basis of the root
lattice of $\widetilde{G}$ with respect to $\widetilde{T}$, the 
kernel of $\phi_{c}$ lies in the center of $\widetilde{G}$. Thus, it is finite.
\end{proof}

Now, let $\sigma$ be the involution of $G\times G$ defined by 
$\sigma(x,y)\,=\,(y,x)$. Note that the diagonal subgroup $\Delta(G)$
of $G\times G$ is the subgroup of fixed points, 
$T\times T$ is a $\sigma$-stable maximal torus of $G\times G$ and
$B\times B^{-}$ is a Borel subgroup having the property that 
$\sigma(\alpha)\,\in\, -R^{+}(B\times B^{-})$
for every $\alpha \,\in\, R^{+}(B\times B^{-}).$

Let $\overline{G}$ denote the wonderful compactification of the group $G$, where $G$ is 
identified with the symmetric space $(G\times G)/\Delta(G)$.

Now, consider the action of $\widetilde{G}$ on $\widetilde{G}$ by conjugation. 
We note that $\widetilde{T}$ is stable under the action of $N_{\widetilde{G}}
(\widetilde{T}).$

It is proved in \cite{St1} that the restriction 
$$
\mathbb{C}[\widetilde{G}]^{\widetilde{G}}\,\longrightarrow \,
\mathbb{C}[\widetilde{T}]^{W(G,T)}
$$
is an isomorphism, and the latter is a polynomial ring.
Hence we have the Steinberg map
$$
\tau\,:\, \widetilde{G}\,\longrightarrow\, \widetilde{T}/W(G,T)\,=\,{\mathbb A}^{n}\, .
$$ 

Let $F$ be the fiber of the Steinberg map $\tau$ containing a representative $n_{c}$ of 
$c$ in $ N_{\widetilde{G}}(\widetilde{T})$. By an abuse of notation, 
we denote by $n_{c}$ the image of  $n_{c}$ in $N_{G}(T)$. Let $F^{\prime}$ be the image 
of $F$ in $G,$ and let $\mathcal{Z} \,= \,Z_1\cup Z_2,$ where $Z_1$ is 
the closure of $F^{\prime}$ in $\overline G$ and $Z_2$ is the unique 
closed $G\times G$ orbit in $\overline G.$

\begin{thm}\label{thm2}
Let $E$ be a $\widetilde{G}\times \widetilde{G}$--equivariant vector bundle on
$\overline G$. Then, $E$ is nef (respectively, ample) if and only if the 
restriction of $E$ to $\mathcal{Z}$ is nef (respectively, ample). Similarly,
$E$ is trivial if and only if its restriction to $\mathcal{Z}$ is so.
\end{thm}

\begin{proof}
Set $W \,=\, W(G,T).$ By the proof of Theorem \ref{thm1}, it is sufficient to prove
that every $T\times T$ stable curve in $\overline G$ lies in $(W\times W)\cdot \mathcal{Z}$. 
It is easy to see that, for every root $\alpha\,\in\, R^{+}(B)$, the 
$T\times T$ stable curve $\overline{({\{1\}\times U_{-\alpha})}\cdot(1,s_{\alpha}) \cdot z}$ lies in 
$Z_{2}$. Similarly, $\overline{(U_{\alpha}\times \{1\}) \cdot 
(s_{\alpha},1)\cdot z}$ lies in $Z_{2}$ for every $\alpha\,\in\, R^{+}(B)$. 
Thus, every $T\times T$ stable curve of type 1 in Lemma \ref{lem5} lies in 
$(W\times W)\cdot \mathcal{Z}$. 

On the other hand, by Lemma \ref{lem7}, the homomorphism $\phi_{c}$ is onto
and hence, the closure of $Tn_{c}\,=\,\{tn_{c}t^{-1}\,\mid\, t\in T\}$ in $\overline G$ is
contained in $Z_{1}$. Therefore every $T\times T$ stable curve
of type 2 in Lemma \ref{lem5} as well 
lies in $(W\times W)\cdot \mathcal{Z}$. This completes the proof of the theorem.
\end{proof}

\section*{Acknowledgements}
We thank the referee for helpful comments. The first--named author thanks the Institute
of Mathematical Sciences for hospitality while this work was carried out. He also
acknowledges the support of the J. C. Bose Fellowship. The second--named author
would like to thank the Infosys Foundation for the partial support.


\begin{thebibliography}{ZZZZ}

\bibitem[Br]{Br} M. Brion, Construction of equivariant vector bundles, 
\textit{Algebraic groups and homogeneous spaces}, 83--111, 
Tata Inst. Fund. Res., Mumbai, 2007. 

\bibitem[BJ]{BJ} M. Brion and R. Joshua, Equivariant Chow ring and 
Chern classes of wonderful symmetric varieties of minimal rank, 
\textit{Transform. Groups} \textbf{13} (2008), 471--493.

\bibitem[DP]{DP} C. De Concini and C. Procesi, Complete symmetric varieties,
\textit{Invariant theory} Montecatini, 1982, 1--44, 
Lecture Notes in Math., 996, Springer, Berlin, 1983.

\bibitem[DPS]{DPS} J.-P. Demailly, T. Peternell and M. Schneider, Compact
complex manifolds with numerically effective tangent bundles, 
\textit{Jour. Alg. Geom.} \textbf{3} (1994), 295--345.

\bibitem[Ha]{Ha} R. Hartshorne, Ample vector bundles,
\textit{Inst. Hautes \'Etudes Sci. Publ. Math.} \textbf{29} (1966), 63--94.

\bibitem[HMP]{HMP} M. Hering, M. Mustata and S. Payne, Positivity
properties of toric vector bundles, \textit{Ann. Inst. Fourier}
\textbf{60} (2010), 607--640.

\bibitem[Hu]{Hu} J. E. Humphreys, \textit{Introduction to Lie algebras and 
Representation theory}, Springer-Verlag, Berlin Heidelberg, New York, 1972.

\bibitem[Hu1]{Hu1} J. E. Humphreys, \textit{Linear Algebraic Groups 
Representation theory}, GTM 21, Springer-Verlag, Berlin Heidelberg, New York, 1975.

\bibitem[Kn1]{Kn} F. Knop, Weylgruppe und Momentabbiidung, \textit{Invent.
Math.} \textbf{99} (1990), 1--23.

\bibitem[Kn2]{Kn1} F. Knop, On the set of orbits for a Borel subgroup, 
\textit{Comment. Math. Helv.} \textbf{70} (1995), 285--309.

\bibitem[Pa]{Pa} D.I. Panyushev, Complexity and rank of homogeneous
spaces, {\it Dokl. Akad. Nauk SSSR} {\bf 307} (1989), 276--279.

\bibitem[Re]{Re} N. Ressayre, Spherical homogeneous spaces of minimal rank, 
\textit{Adv. Math.} \textbf{224} (2010), 1784--1800

\bibitem[Ri]{Ri} R. W. Richardson, Orbits, invariants, and representations 
associated to involutions of reductive groups, 
\textit{Invent. Math.} \textbf{66} (1982), 287--312.

\bibitem[RS]{RS} R. W. Richardson and T. A. Springer, The Bruhat order on symmetric
varieties, \textit{Geom. Dedicata} \textbf{35} (1990), 389--436.

\bibitem[Se]{Se} J.-P. Serre, On the fundamental group of a 
unirational variety, \textit{Jour. Lond. Math. Soc.} \textbf{34} (1959), 481--484.

\bibitem[Si]{Si} C. T. Simpson, Higgs bundles and local systems, \textit{Inst.
Hautes \'Etudes Sci. Publ. Math.} \textbf{75} (1992), 5--95.

\bibitem[St1]{St1} R. Steinberg, Regular elements of semisimple algebraic
groups, \textit{Inst. Hautes \'Etudes Sci. Publ. Math.} \textbf{25} (1965), 49--80.

\bibitem[St2]{St2} R. Steinberg, Endomorphisms of linear algebraic groups,
\textit{Memoirs of the American Mathematical Society}, No. 80, 1968.

\bibitem[Tc]{Tc} A. Tchoudjem, Cohomologie des fibr\'{e}s en droites sur les 
vari\'{e}t\'{e}s magnifiques de rang minimal, \textit{Bull. Soc.
Math. Fr.} \textbf{135} (2007), 171--214. 

\bibitem[YZ]{YZ} S. W. Yang and A. Zelevinsky, Cluster algebras of finite type 
via Coxeter elements and principal minors, 
\textit{Transform. Groups} \textbf{13} (2008), 855--895. 

\end{thebibliography}
\end{document}